\documentclass[10pt]{amsart}
\usepackage{amsrefs}
\usepackage{amssymb}
\usepackage[all]{xy}

\usepackage{hyperref}

\usepackage{tikz}
\usetikzlibrary{matrix,arrows}

\DeclareMathOperator{\Img}{Im}

\DeclareMathOperator{\Zen}{Z}

\DeclareFontFamily{U}{wncy}{}
\DeclareFontShape{U}{wncy}{m}{n}{<->wncyr10}{}
\DeclareSymbolFont{mcy}{U}{wncy}{m}{n}
\DeclareMathSymbol{\Sha}{\mathord}{mcy}{"58}
\DeclareMathSymbol{\sha}{\mathord}{mcy}{"78}

\begin{document}

\newtheorem{thm}{Theorem}[section]
\newtheorem{cor}[thm]{Corollary}
\newtheorem{lem}[thm]{Lemma}
\newtheorem{fact}[thm]{Fact}
\newtheorem{prop}[thm]{Proposition}
\newtheorem{defin}[thm]{Definition}
\newtheorem{exam}[thm]{Example}
\newtheorem{examples}[thm]{Examples}
\newtheorem{rem}[thm]{Remark}
\newtheorem{case}{\sl Case}
\newtheorem*{claim}{Claim}
\newtheorem{question}[thm]{Question}
\newtheorem{conj}[thm]{Conjecture}
\newtheorem*{notation}{Notation}
\swapnumbers
\newtheorem{rems}[thm]{Remarks}
\newtheorem*{acknowledgment}{Acknowledgments}
\newtheorem*{thmno}{Theorem}

\newtheorem{questions}[thm]{Questions}
\numberwithin{equation}{section}

\newcommand{\gr}{\mathrm{gr}}
\newcommand{\inv}{^{-1}}
\newcommand{\isom}{\cong}
\newcommand{\dbC}{\mathbb{C}}
\newcommand{\F}{\mathbb{F}}
\newcommand{\dbN}{\mathbb{N}}
\newcommand{\Q}{\mathbb{Q}}
\newcommand{\dbR}{\mathbb{R}}
\newcommand{\dbU}{\mathbb{U}}
\newcommand{\Z}{\mathbb{Z}}
\newcommand{\calG}{\mathcal{G}}
\newcommand{\calX}{\mathcal{X}}
\newcommand{\K}{\mathbb{K}}
\newcommand{\rmH}{\mathrm{H}}
\newcommand{\bfH}{\mathbf{H}}
\newcommand{\rmr}{\mathrm{r}}
\newcommand{\Span}{\mathrm{Span}}
\newcommand{\eue}{\mathbf{e}}

\newcommand{\bfLam}{\mathbf{\Lambda}}
\newcommand{\calY}{\mathcal{Y}}
\newcommand{\calV}{\mathcal{V}}
\newcommand{\calE}{\mathcal{E}}
\newcommand{\calW}{\mathcal{W}}


\title[Digraphs, Frattini-resistance and pro-$p$ Galois groups]{Directed graphs, Frattini-resistance, \\ and maximal pro-$p$ Galois groups}

\author{Claudio Quadrelli}
\address{Department of Science \& High-Tech, University of Insubria, Como, Italy EU}
\email{claudio.quadrelli@uninsubria.it}
\date{\today}

\begin{abstract}
Let $p$ be a prime.
Following Snopce-Tanushevski, a pro-$p$ group $G$ is called Frattini-resistant if the function
$H\mapsto\Phi(H)$, from the poset of all closed topologically finitely generated subgroups of $G$ into itself, is a poset embedding.
We prove that for an oriented right-angled Artin pro-$p$ group (oriented pro-$p$ RAAG) $G$ associated to a finite directed graph the following four conditions are equivalent: the associated digraph is of elementary type; $G$ is Frattini-resistant; every topologically finitely generated closed subgroup of $G$ is an oriented pro-$p$ RAAG; $G$ is the maximal pro-$p$ Galois group of a field containing a root of 1 of order $p$.
Also, we conjecture that in the $\Z/p$-cohomology of a Frattini-resistant pro-$p$ group there are no essential triple Massey products.
\end{abstract}



\maketitle

\section{Introduction}
\label{sec:intro}

Let $p$ be a prime number.
In the paper \cite{ilirslobo:fratres}, Ilir~Snopce and Slobodan~Tanushevski introduced the notion of {\sl Frattini-resistant pro-$p$ group}.
Given a pro-$p$ group $G$, let $\Phi(G)$ denote the Frattini subgroup of $G$.
Then $G$ is said to be Frattini-resistant if the following condition is satisfied: for every pair of topologically finitely generated closed subgroups $H_1$ and $H_2$ of $G$, one has
\[
 \Phi(H_1)\subseteq \Phi(H_2) \qquad\text{if, and only if,}\qquad H_1\subseteq H_2
\]
(cf. \cite[\S~4]{ilirslobo:fratres}).
Besides being interesting on its own, this property is particularly relevant in Galois theory, as the {\sl maximal pro-$p$ Galois group} of a field containing a root of 1 of order $p$ is Frattini-resistant --- as shown by Snopce-Tanushevski (cf. \cite[\S~7]{ilirslobo:fratres}), and as we will recall below.

Our aim is to study Frattini-resistance for {\sl oriented right-angled Artin pro-$p$ groups} associated to {\sl directed graphs}.
By a {\sl directed graph} (or {\sl digraph}, for short) $\Gamma$ we mean a pair of sets $\Gamma=(\calV,\calE)$ --- we tacitly assume that $\calV\cap\calE=\varnothing$, and that $\calV$ is finite --- where $\calV$ is the set of {\sl vertices} of $\Gamma$, and $\calE$ is the set of {\sl edges} of $\Gamma$, with
$$\calE\subseteq \calV\times\calV\smallsetminus\{\:(v,v)\:\mid\:v\in\calV\:\}$$
(see, e.g., \cite[\S~1.10]{diestel}).
Given a pair of vertices $v,w\in\calV$, if $(w,v)\in\calE$ but $(v,w)\notin \calE$ the edge $(w,v)$ is said  a {special edge}, while if both
$(w,v),(v,w)$ lie in $\calE$, we say that $(w,v)$ and $(v,w)$ are {ordinary edges}.
A vertex which is the second coordinate of a special edge is said to be {\sl special}.
For example, the diagrams
\begin{equation}\label{eq:exampleintro}
 \xymatrix@R=2pt{ v_1 && v_2\\ \bullet && \circ\ar[ll]\ar@/^/[dddd]
 \\ \\ \\ \\
 \circ\ar[uuuu]\ar@/_/[rr] & & \circ\ar[uuuull]\ar@/^/[uuuu]\ar@/_/[ll]  \\ v_4 && v_3}
 \qquad\qquad\text{and}\qquad\qquad
 \xymatrix@R=2pt{ v_1 && v_2\\\bullet && \circ\ar[ll]\ar@{-}[dddd]
 \\ \\ \\ \\
 \circ\ar[uuuu]\ar@{-}[rr]  && \circ\ar[uuuull] \\ v_4 && v_3}
 \end{equation}
represent the same digraph with four vertices: in the first diagram every edge is represented as an arrow going from the first coordinate to the second, while in the second diagram we identify the ordinary edges joining the same pair of vertices, and we represent them as a single ``headless'' edge, so that only the special edges are represented as arrows.
Moreover, the special vertex is the black one.

The {\sl oriented right-angled Artin pro-$p$ group} (oriented pro-$p$ RAAG for short) associated to a digraph $\Gamma=(\calV,\calE)$ and to a $p$-power $q=p^f$ with $f\geq1$ --- and $f\geq 2$ if $p=2$ --- is the pro-$p$ group with pro-$p$ presentation
\[
 G=\left\langle\:v\in\calV\:\mid\: wuw^{-1}=\begin{cases}
                                             u^{1+q}& \text{if }(w,u)\text{ is special}, \\
                                             u& \text{if }(w,u)\text{ is ordinary},
        \end{cases}\;\forall\; (u,w)\in\calE \:\right\rangle.
\]

The family of oriented pro-$p$ RAAGs is very rich: inside it, one may find all finitely generated free --- and free abelian --- pro-$p$ groups, pro-$p$ completions of discrete RAAGs, certain families of $p$-adic analytic pro-$p$ groups and even some finite $p$-groups (see, e.g., \cite[\S~1]{BQW}).
Recently, oriented pro-$p$ RAAGs have been object of study, especially from a Galois-theoretic perspective (see, e.g., \cite{BCQ,BQW,CasQua,qsv:quadratic,sz:raags}).
Our main goal is to characterize the oriented pro-$p$ RAAGs which are Frattini-resistant, in terms of the associated digraph.

\begin{thm}\label{thm:main intro}
 Let $\Gamma=(\calV,\calE)$ be a digraph, and for $q=p^f$ {\rm(}with $f\geq2$ if $p=2${\rm)} let $G$ be the oriented pro-$p$ RAAG associated to $\Gamma$ and to $q$.
 The following are equivalent.
 \begin{itemize}
  \item[(i)] $\Gamma$ is a digraph of elementary type.
  \item[(ii)] $G$ is Frattini-resistant.
 \end{itemize}
\end{thm}

A digraph $\Gamma=(\calV,\calE)$ is said to be {\sl of elementary type} if it satisfies the following conditions: every edge joining a special vertex and another vertex is an arrow pointing at the special vertex --- roughly speaking, every special vertex is a ``sinkhole'' --- (e.g., the only special vertex, $v_1$, in the digraph \eqref{eq:exampleintro}, is a sinkhole); and for every induced subgraph $\Gamma'$ of $\Gamma$, either
\begin{itemize}
 \item[(a)] $\Gamma'$ has at least two distinct connected components,
 \item[(b)] or $\Gamma'$ has an ordinary vertex which is adjacent to all other vertices of $\Gamma'$.
\end{itemize}
As an example, the digraph represented in \eqref{eq:exampleintro} is not of elementary type, as the induced subgraph with vertices $v_1,v_2,v_4$ satisfies none of the conditions (a)--(b) above.
Also, consider the two digraphs with geometric representations
\begin{equation}\label{eq:examples graphs square}
 \xymatrix@R=1.5pt{v_1&v_2 \\ \bullet & \circ\ar[l]\ar[ddd]  \\ \\ \\
 \circ\ar[uuu]\ar[r]\ar@{-}[uuur]  & \bullet\\ v_4&v_3}
 \qquad\qquad\text{and}\qquad\qquad
 \xymatrix@R=1.5pt{v_1&v_2 \\ \bullet\ar@{-}[r] & \circ\ar@{-}[ddd]  \\ \\ \\
 \circ\ar[uuu]\ar@{-}[r]\ar@{-}[uuur]  & \circ\\ v_4&v_3}
 \end{equation}
The first digraph is of elementary type, while the second is not as the special vertex $v_1$ is not a sinkhole.
Digraphs of elementary type may be constructed starting from single vertices as ``bricks'', and performing elementary operations, namely disjoint unions and ``coning'' (see \cite[\S~2.4]{BQW} and \S~\ref{ssec:et graph} below).

From Theorem~\ref{thm:main intro} we may deduce some interesting corollaries.
The first one may be seen as the translation, in terms of Frattini-resistance, of the ``oriented pro-$p$ version'' of a famous result of C.~Droms on discrete RAAGs (see \cite{droms}).

\begin{cor}\label{cor:droms}
 Let $\Gamma=(\calV,\calE)$ be a digraph, and for $q=p^f$ {\rm(}with $f\geq2$ if $p=2${\rm)} let $G$ be the oriented pro-$p$ RAAG associated to $\Gamma$ and to $q$.
 The following are equivalent.
 \begin{itemize}
  \item[(i)] $G$ is Frattini-resistant.
  \item[(ii)] Every topologically finitely generated closed subgroup of $G$ occurs as an oriented pro-$p$ RAAG associated to some digraph.
 \end{itemize}
\end{cor}

The proof of Corollary~\ref{cor:droms} relies on the characterization of oriented pro-$p$ RAAGs satisfying ``Drom's condition'' (ii) proved by S.~Blumer, Th.S.~Weigel and the author in \cite{BQW}.

The second corollary we obtain from Theorem~\ref{thm:main intro} involves the realizability of an oriented pro-$p$ RAAG as the {\sl maximal pro-$p$ Galois group} of a field.
Given a field $\K$, its maximal pro-$p$ Galois group $G_{\K}(p)$ is the Galois group of the maximal pro-$p$ Galois extension of $\K$ --- or, equivalently, the maximal pro-$p$ quotient of the absolute Galois group of $\K$.
Characterizing those pro-$p$ groups which occur as the maximal pro-$p$ Galois group of a filed is one of the main open problems in Galois theory (see, e.g., \cites{BLMS,BQW,cq:noGal,cq:chase}).

Recall that a pro-$p$ group $G$ is said to complete into a {\sl 1-cyclotomic oriented pro-$p$ group} if there exists a continuous $G$-module $M$, isomorphic to $\Z_p$ as an abelian pro-$p$ group, such that for every closed subgroup $H$ of $G$, the natural cohomology maps
\begin{equation}\label{eq:1cyc}
 \rmH^1(H,M/p^nM)\longrightarrow \rmH^1(H,M/pM),
\end{equation}
induced by the epimorphism of continuous $H$-modules $M/p^nM\twoheadrightarrow M/pM$, is surjective for every $n\geq1$ --- here we consider the continuous $G$-modules $M/p^nM$ as continuous $H$-modules in the obvious way ---, see, e.g., \cite[\S~1]{cq:chase}.
It is well-known that the maximal pro-$p$ Galois group of a field containing a root of 1 of order $p$, together with the {\sl pro-$p$ cyclotomic character}, completes into a 1-cyclotomic oriented pro-$p$ group (see Example~\ref{exam:GK} below).

By \cite[Thm.~1.1]{BQW}, an oriented pro-$p$ RAAG completes into a 1-cyclotomic oriented pro-$p$ group if, and only if, the associated digraph is of elementary type.
Also, Snopce-Tanushevski proved that a pro-$p$ group which completes into a 1-cyclotomic oriented pro-$p$ group is Frattini-resistant (see \cite[Thm.~1.11]{ilirslobo:fratres}).
Hence, combining these two results with Theorem~\ref{thm:main intro} yields the following.

\begin{cor}\label{cor:1cyc}
 Let $\Gamma=(\calV,\calE)$ be a digraph, and for $q=p^f$ {\rm(}with $f\geq2$ if $p=2${\rm)} let $G$ be the oriented pro-$p$ RAAG associated to $\Gamma$ and to $q$.
 The following are equivalent.
 \begin{itemize}
  \item[(i)] $G$ is Frattini-resistant.
  \item[(ii)] $G$ completes into a 1-cyclotomic oriented pro-$p$ group.
  \item[(iii)] $G$ occurs as the maximal pro-$p$ Galois group of a field containing a root of $1$ of order $p$.
 \end{itemize}
\end{cor}

Corollary~\ref{cor:1cyc} suggests that Frattini-resistance is a rather restrictive property for pro-$p$ groups, and it may provide a powerful tool to study maximal pro-$p$ Galois groups of fields.
In particular, we believe that further investigations in this direction will lead to the discovery of new obstructions for the realization of pro-$p$ groups as maximal pro-p Galois group.

Another cohomological property of maximal pro-$p$ Galois groups which has been thoroughly studied in recent years is the presence of {\sl essential Massey products} in the $\Z/p$-cohomology of these pro-$p$ groups.
Massey products in the $\Z/p$-cohomology of a pro-$p$ group $G$ are multi-valued maps which associate a sequence of elements of $\rmH^1(G,\Z/p)$ to a subset of $\rmH^2(G,\Z/p)$, and they generalize the cup-product:
for an overview on Massey products in the $\Z/p$-cohomology of maximal pro-$p$ Galois groups see, e.g., \cite{MT:kernel,MT:massey} and references therein.
E.~Matzri proved that if $\K$ is a field containing a root of 1 of order $p$, then there are no essential triple Massey products in the $\Z/p$-cohomology of $G_{\K}(p)$ --- see \cite{eli:Massey,EM:Massey} ---, i.e., whenever the subset of $\rmH^2(G_{\K}(p),\Z/p)$ which is the value of a Massey product of a sequence of length 3 of elements of $\rmH^1(G_{\K}(p),\Z/p)$ is not empty, it contains 0.
Employing a result of W.~Dwyer, it is possible to translate this property into purely group-theoretic terms, see \S~\ref{ssec:massey unip} below.
This property was used to produce new examples of pro-$p$ groups that do not occur as maximal pro-$p$ Galois groups of fields containing a root of 1 of order $p$ (see \cite[\S~7]{MT:massey}).

We suspect that Frattini-resistance is strictly stronger than this cohomological property.

\begin{conj}\label{conj:massey}
Let $G$ be a Frattini-resistant pro-$p$ group.
Then there are no essential triple Massey products in the $\Z/p$-cohomology of $G$.
\end{conj}

In \S~\ref{ssec:massey conj} we formulate Conjecture~\ref{conj:massey} in group-theoretic terms.


\section{Frattini-resistance and 1-cyclotomicity}\label{sec:fratres}

\subsection{Frattini-resistant pro-$p$ groups}\label{ssec:fratres}

Let $G$ be a pro-$p$ group.
From now on, every subgroup of $G$ will be tacitly assumed to be closed with respect to the pro-$p$ topology, and sets of generators, and presentations, of pro-$p$ groups will be intended in the topological sense.
Also, all morphisms between pro-$p$ groups --- including morphisms from a pro-$p$ group to a finite $p$-group --- will be tacitly assumed to be continuous.

Given two elements $x,y\in G$, we adopt the notation
\[
 {}^xy=x\cdot y\cdot x^{-1}\qquad\text{and}\qquad[x,y]={}^xy\cdot y^{-1}=xyx^{-1}y^{-1}.
\]
Given a positive integer $n$, one has the normal subgroups
\[ G^{p^n} =\left\langle\: x^{p^n}\:\mid\:x\in G\:\right\rangle\qquad\text{and}\qquad
 G' =\left\langle\:[x,y]\:\mid\:x,y\in G\:\right\rangle.
 \]
Finally, the {\sl Frattini subgroup} of $G$ is $\Phi(G)=G^p\cdot G'$.
The quotient $G/\Phi(G)$ is a $\Z/p$-vector space, and a basis of this quotient yields a minimal generating set of $G$ (cf., e.g., \cite[Prop.~1.9]{ddsms}).
We remark that if $G$ is a finitely generated pro-$p$ group, then also $\Phi(G)$ is finitely generated, and the dimension of $G/\Phi(G)$ is finite (cf., e.g., \cite[Prop.~1.14 and Thm.~1.17]{ddsms}).

Thus, one may formulate the definition of Frattini-resistant pro-$p$ group as follows: $G$ is Frattini-resistant if the assignment $H\mapsto \Phi(H)$ yields a homomorphism of partially ordered sets from the partially ordered set of finitely generated subgroups of $G$ into itself (cf. \cite[Def.~1.1]{ilirslobo:products}).
Another notion introduced by Snopce-Tanushevski, tightly related to Frattini-resistance, is the following (cf. \cite[Def.~1.1]{ilirslobo:fratres}.

\begin{defin}\rm
A pro-$p$ group $G$ is said to be {\sl Frattini-injective} if, given two finitely generated subgroups $H_1,H_2$ of $G$, $\Phi(H_1)=\Phi(H_2)$ implies that $H_1=H_2$.
\end{defin}

It is easy to see that if $G$ is Frattini-resistant, then it is also Frattini-injective (cf. \cite[Cor.~4.3]{ilirslobo:fratres}).
A first, easily checked, property enjoyed by Frattini-resistant (and Frattini-injective) pro-$p$ groups is that they are torsion-free (cf. \cite[\S~1]{ilirslobo:fratres}).

\begin{exam}\label{ex:torsion}\rm
 Let $G$ be a pro-$p$ group, and suppose that $g\neq1$ is an element of $G$ yielding non-trivial torsion, i.e., $g^{p^k}=1$ and $g^{p^{k-1}}\neq1$ for some $k\geq1$.
 Then
 $$1=g^{p^k}=\left(g^{p^{k-1}}\right)^p\in\Phi(\{1\})=\{1\}\qquad\text{but}\qquad
 1\neq g^{p^{k-1}}\notin \{1\},$$
 and therefore $G$ is not Frattini-resistant.
 Also, it is not Frattini-injective, as $\Phi(\langle g^{p^{k-1}}\rangle)$ and $\Phi(\{1\})$ are equal.
\end{exam}

One has the following handy characterization of Frattini-resistant pro-$p$ groups (cf. \cite[Prop.~2.1]{ilirslobo:products}).
For the reader's convenience, we recall briefly its proof, as presented by Snopce-Tanushevski.

\begin{prop}\label{prop:fratres commres}
 Let $G$ be a pro-$p$ group.
 The following are equivalent:
 \begin{itemize}
  \item[(i)] $G$ is Frattini-resistant;
  \item[(ii)] if an element $x\in G$ satisfies $x^p\in\Phi(H)$ for some finitely generated subgroup $H\subseteq G$, then $x\in H$.
 \end{itemize}
\end{prop}

\begin{proof}
 Suppose that $G$ is Frattini-resistant, and let $x\in G$ and $H\subseteq G$ be such that $x^p\in\Phi(H)$. Then $\Phi(\langle x\rangle)=\langle x^p\rangle\subseteq \Phi(H)$, and hence $\langle x\rangle\subseteq H$.

 Conversely, suppose that condition~(ii) is satisfied, and let $H_1,H_2$ be two subgroups of $G$ such that $\Phi(H_1)\subseteq \Phi(H_2)$.
 Then for every $x\in H_1$ one has $x^p\in \Phi(H_1)\subseteq \Phi(H_2)$, and therefore $x\in H_2$ by condition~(ii). Thus $H_1\subseteq H_2$.
\end{proof}

\begin{exam}\label{ex:freeabel}\rm
 A free abelian pro-$p$ group $G$ is Frattini-resistant.
 Indeed, for any finitely generated subgroup $H\subseteq G$, one has
 \[
  \Phi(H)=H^p=\left\{\:h^p\:\mid\:h\in H\:\right\},
 \]
and hence if $g^p\in \Phi(H)$, for $g\in G$, then $g^p=h^p$ for some $h\in H$.
Thus, $1=g^ph^{-p}=(gh^{-1})^p$, which implies that $g=h$ as $G$ is torsion-free.
\end{exam}

\begin{exam}\label{ex:prop Heisenberg}\rm
The {\sl Heisenberg pro-$p$ group}
 \[\begin{split}
 G &=\left\langle\:x,\:y,\:z\:\mid\:[x,y]=z,\:[x,z]=[y,z]=1\right \rangle \\
 &=\left\langle\:x,\:y\:\mid\:[x,[x,y]]=[[y,[x,y]]=1 \right\rangle
\end{split}
\]
is not Frattini-resistant.
Indeed, let $H$ be the subgroup of $G$ generated by $x,y^p$.
Then $[x,y^p]=z^p$, and
\[
 H=\left\langle\:x,\:y^p,\:z^p\:\mid\:\left[x,y^p\right]=z^p,\:[x,z^p]=[y^p,z^p]=1\right \rangle,
\]
so that $z^p\in \Phi(H)$, but $z\notin H$.
\end{exam}

One may employ direct products with free abelian pro-$p$ groups to produce new Frattini-resistant pro-$p$ groups, but only under certain restrictions: the following proposition is due to Snopce-Tanushevski, cf. \cite[Thm.~A]{ilirslobo:products} --- we recall that a pro-$p$ group $G$ is said to be {\sl absolutely torsion-free} if for every subgroup $H$, the abelianization $H/H'$ is a free abelian pro-$p$ group, cf. \cite{wurfel}.

\begin{prop}\label{prop:directproducts}
 Let $G_1,G_2$ be two non-trivial pro-$p$ groups, and set $G=G_1\times G_2$.
 Then $G$ is Frattini-resistant if, and only if, and only if, both $G_1,G_2$ are absolutely torsion-free, and at least one of the two factors is a free abelian pro-$p$ group.
\end{prop}

\subsection{1-cyclotomic oriented pro-$p$ groups}\label{ssec:1cyc}
Let $1+p\Z_p$ denote the multiplicative group of principal units of the ring of $p$-adic integers $\Z_p$ --- i.e.,
\[
 1+p\Z_p=\{\:1+p\lambda\:\mid\:\lambda\in\Z_p\:\}.
\]
If $p\neq2$ then $1+p\Z_p\simeq\Z_p$ as an abelian pro-$p$ group.

Given a pro-$p$ group $G$, an {\sl orientation} of $G$ is a homomorphism $\theta\colon G\to 1+p\Z_p$, and the pair $(G,\theta)$ is called an {\sl oriented pro-$p$ group} (cf. \cite{qw:cyc}; oriented pro-$p$ groups were introduced by I.~Efrat in \cite{efrat:small} with the name ``cyclotomic pro-$p$ pairs'').
An oriented pro-$p$ group $(G,\theta)$ comes endowed with a canonical continuous $G$-module $\Z_p(\theta)$, which is isomorphic to $\Z_p$ as an abelian pro-$p$ group, with the action given by
\[
 g.\lambda=\theta(g)\cdot \lambda\qquad\text{for every }g\in G,\:\lambda\in\Z_p(\theta).
\]
Conversely, if $G$ is a pro-$p$ group and $M$ is a continuous $G$-module, isomorphic to $\Z_p$ as an abelian pro-$p$ group, then the action of $G$ on $M$ induces an orientation $\theta\colon G\to1+p\Z_p$ by $\theta(g)m=g.m$ for every $m\in M$, and $M\simeq\Z_p(\theta)$ as a continuous $G$-module.

Now consider the epimorphisms of continuous $G$-modules
$$\xymatrix{\Z_p(\theta)/p^n\Z_p(\theta)\ar@{->>}[r] & \Z_p(\theta)/p\Z_p(\theta)},$$
induced by the canonical projections $\Z/p^n\twoheadrightarrow\Z_p$, for every $n\geq1$ --- observe that the $G$-module $\Z_p(\theta)/p\Z_p(\theta)$ is isomorphic to the trivial $G$-module $\Z/p$.
These epimorphisms induce in cohomology the natural maps
\begin{equation}\label{eq:def 1cyc}
 \rmH^1(G,\Z_p(\theta)/p^n\Z_p(\theta))\longrightarrow \rmH^1(G,\Z/p).
\end{equation}
An oriented pro-$p$ group $(G,\theta)$ is said to be {\sl 1-cyclotomic} if the maps \eqref{eq:def 1cyc} are surjective for every $n\geq1$, and also for every subgroup $H$ of $G$, replacing $G$ with $H$ and  $\Z_p(\theta)$ with $\Z_p(\theta\vert_H)$ in \eqref{eq:def 1cyc}.

\begin{rem}\label{rem:H1}\rm
 The cohomology group of degree 1, $\rmH^1(G,\Z/p)$, is the group of homomorphisms of pro-$p$ groups $G\to \Z/p$.
 Hence, one has an isomorphism of discrete $\Z/p$-vector spaces $\rmH^1(G,\Z/p)\simeq(G/\Phi(G))^\ast$, where the latter is the $\Z/p$-dual of $G/\Phi(G)$ (cf., e.g., \cite[Chap.~I, \S~4.2]{serre:galc}).
\end{rem}

We say that a pro-$p$ group $G$ may complete into a 1-cyclotomic oriented pro-$p$ group if there exists an orientation $\theta\colon G\to1+p\Z_p$ such that the oriented pro-$p$ group $(G,\theta)$ is 1-cyclotomic.
In \cite[Thm.~1.11]{ilirslobo:fratres}, I.~Snopce and S.~Tanushevski prove that pro-$p$ groups which may complete into 1-cyclotomic oriented pro-$p$ groups are Frattini-resistant (with a condition if $p=2$).

\begin{prop}\label{prop:1cyc fratres}
 Let $(G,\theta)$ be a 1-cyclotomic torsion-free oriented pro-$p$ group, and suppose that $\Img(\theta)\subseteq1+4\Z_2$ if $p=2$.
 Then $G$ is strongly Frattini-resistant.
\end{prop}

One of the most relevant examples of 1-cyclotomic oriented pro-$p$ groups --- and of Frattini-resistant pro-$p$ groups --- is provided by Galois theory (cf. \cite[\S~4]{eq:kummer} and \cite[\S~2.4]{cq:chase}).

\begin{exam}\label{exam:GK}\rm
Let $\K$ be a field containing a root of 1 of order $p$ (and also $\sqrt{-1}$ if $p=2$), and consider its maximal pro-$p$ Galois group $G_{\K}(p)$.
The {\sl pro-$p$ cyclotomic character} $\theta_{\K}$ of $G_{\K}(p)$ is the orientation $\theta_{\K}\colon G_{\K}(p)\to1+p\Z_p$ satisfying
\[
 g(\zeta)=\zeta^{\theta_{\K}(g)}\qquad\text{for every }g\in G\text{ and }\zeta\in\bar\K_s^\times\text{ of order }p^k
\]
for some $k\geq1$ (here $\bar\K_s$ denotes the separable closure of $\K$).
Since we are assuming that $\sqrt{-1}\in\K$ if $p=2$, then in this case $\Img(\theta_{\K})\subseteq1+4\Z_2$.
The oriented pro-$p$ group $(G_{\K}(p),\theta_{\K})$ is 1-cyclotomic, and therefore the maximal pro-$p$ Galois group $G_{\K}(p)$ is Frattini-resistant.
\end{exam}


\section{Digraphs}\label{sec:orgraphs}

\subsection{Digraphs and special digraphs}\label{ssec:special}

Let $\Gamma=(\calV,\calE)$ be a digraph.
Recall that the vertices that are the second coordinate of a special edge are special (and we represent them with black bullets); conversely, the vertices that are not special are called ordinary vertices (and we represent them with white bullets).
One has the following notions for digraphs.
\begin{itemize}
 \item[(a)] An {\sl induced subdigraph} of $\Gamma$ is a digraph $\Gamma'=(\calV',\calE')$ such that $\calV'\subseteq \calV$, and $\calE'=\calE\:\cap\:(\calV'\times\calV')$;
and moreover, a vertex $v\in\calV'$ is special, respectively ordinary, whenever it is a special, respectively ordinary, vertex of $\Gamma$.
\item[(b)] A special vertex $w\in\calV$ is called a {\sl sinkhole} if $(u,w)$ is a special edge of $\Gamma$ whenever $u\in\calV$ is another vertex which is joined to $w$.
\item[(c)] $\Gamma$ is a {\sl special digraph} if every special vertex is a sinkhole.
\end{itemize}

For example, the digraph represented in \eqref{eq:exampleintro} and the left-side digraph in \eqref{eq:examples graphs square} are special digraphs, while the right-side digraph in \eqref{eq:examples graphs square} is not, as $v_1$ is not a sinkhole.

\begin{rem}\label{rem:non or graph}\rm
\begin{itemize}
 \item[(a)] Henceforth, if $\Gamma=(\calV,\calE)$ is a digraph and $(u,v),(v,u)\in\calE$ are ordinary edges, then we will identify them, and we will say that $u,v$ are joined by a single ordinary edge, with an abuse of notation.
Moreover, we will call digraphs without special edges ``undigraphs'' (cf. \cite[Rem.~2.3]{cq:orMassey}).
\item[(b)] In \cite{BQW}, a digraph is --- uncorrectly --- called an ``oriented graph'', while an actual oriented graph is a digraph without ordinary edges (cf., e.g., \cite[p.~28]{diestel}).
Hence, all results of \cite{BQW} on ``oriented graphs'' (and associated oriented pro-$p$ RAAGs), should be read as results on digraphs (and associated oriented pro-$p$ RAAGs).
\end{itemize}

\end{rem}

It is straightforward to see that a digraph is special if, and only if, none of the following occurs:
\begin{equation}\label{eq:subgraphs no special}
 \xymatrix@R=1.5pt{&x&\\ \circledast\ar@/_1.25pc/@{--}[rr]\ar[r] & \bullet\ar[r] & \bullet \\  z &  & y}
 \qquad\text{or}\qquad
\xymatrix@R=1.5pt{&x&\\ \circledast\ar@/_1.25pc/@{--}[rr]\ar[r] & \bullet\ar@{-}[r] & \circledast \\  z &  & y}
 \end{equation}
--- no matter whether the bottom vertices are joined or not (here we use $\circledast$ to represent vertices which are not necessarily ordinary nor special) ---  cf. \cite[\S~2.3]{BQW} or \cite[\S~2.2]{cq:orMassey}.
Altogether, all possible cases of induced subdigraphs with three vertices which prevent a digraph to be special are seven: namely, for each of the
two representations in \eqref{eq:subgraphs no special} one has four possible cases --- $z,y$ are disjoint; or joined by an ordinary edge; or joined by a special edge, either $(z,y)$ or $(y,z)$ ---, and the two representations
\[ \xymatrix@R=1.5pt{&x&\\ \circ\ar@/_1.25pc/@{-}[rr]\ar[r] & \bullet\ar[r] & \bullet \\  z &  & y}
 \qquad\text{and}\qquad
\xymatrix@R=1.5pt{&x&\\ \bullet\ar[r] & \bullet\ar@{-}[r] & \circ\ar@/^1.25pc/[ll] \\  z &  & y}
\]
yield the same case, with the ``roles'' of the vertices $x,y,z$ permuted cyclically (see \S~\ref{sec:nonspecial} below).

\subsection{Digraphs of elementary type}\label{ssec:et graph}
One has the following two operations with digraphs.
\begin{itemize}
 \item[(a)] Given two digraphs $\Gamma_1=(\calV_1,\calE_1)$ and $\Gamma_2=(\calV_2,\calE_2)$, their {\sl disjoint union} is the digraph $\Gamma=(\calV,\calE)$ with
 \[
  \calV=\calV_1\:\dot\cup\:\calV_2\qquad\text{and}\qquad\calE=\calE_1\:\dot\cup\:\calE_2.
 \]
\item[(b)] Given a digraph $\Gamma=(\calV,\calE)$, the {\sl cone} $\nabla(\Gamma)=(\calV_{\nabla(\Gamma)},\calE_{\nabla(\Gamma)})$ of $\Gamma$ is the digraph obtained by adding a ``new'' ordinary vertex to $\Gamma$, and joining it with all ordinary vertices of $\Gamma$, via ordinary edges, and with all special vertices of $\Gamma$ via special edges: namely,
\[
\begin{split}
  \calV_{\nabla(\Gamma)}&=\{\:u\:\}\:\dot\cup\:\calV,\\
\calE_{\nabla(\Gamma)}&=\{\:(u,v),\:(v,u),\:(u,w)\:\mid\:v,w\in\calV,\text{ $v$ ordinary, $w$ special}\:\}\:\dot\cup\:\calE,\\
\end{split}
\]
where $u$ is the ``new'' vertex (which we call the tip of the cone).
\end{itemize}

\begin{defin}\label{defin:ET graph}\rm
 A digraph $\Gamma=(\calV,\calE)$ is of {\sl elementary type} if it may be obtained by iterating disjoint union and cones starting from a subset $\calV_0$ of $\calV$.
\end{defin}

\begin{exam}\rm
 Let $\Gamma=(\calV,\calE)$ be the left-side digraph in \eqref{eq:examples graphs square}.
 Then $\Gamma$ is of elementary type, and it may be constructed as follows: we start with the disjoint vertices $v_1$ and $v_3$, which are special, then we make the cone with the ordinary vertex $v_3$ as the tip, and finally we make the cone of the resulting digraph with the ordinary vertex $v_4$ as tip.
 
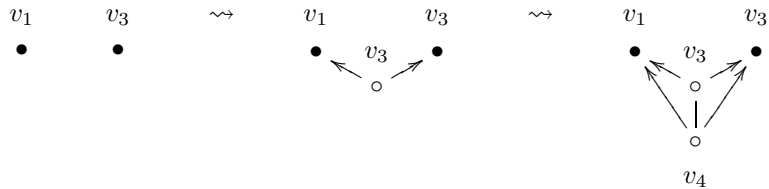
\[
 \xymatrix@R=1.5pt@C=7pt{ v_1 && v_3\\ \bullet&&\bullet}
 \qquad\ \rightsquigarrow\qquad
 \xymatrix@R=1.5pt@C=7pt{ v_1 && v_3\\ \bullet& v_3&\bullet \\ &\circ\ar[ur]\ar[ul]&}
\qquad\ \rightsquigarrow\qquad
 \xymatrix@R=1.5pt@C=7pt{ v_1 && v_3\\ \bullet& v_3&\bullet \\ &\circ\ar[ur]\ar[ul]& \\ \\ &\circ\ar@{-}[uu]\ar[uuul]\ar[uuur]&\\ &v_4&}
 \]
 In this case $\calV_0=\{v_1,v_3\}$.
\end{exam}

\begin{exam}\rm
 Consider the digraph $\Gamma=(\calV,\calE)$ with geometric realization
 \[  \xymatrix@R=1.5pt{ v_1 &  v_2 & v_3 \\ \circ\ar[ddd] & \bullet & \circ\ar[ddd]\ar@{-}[dddl] \\ \\ \\
 \bullet&\circ\ar[uuu]\ar[l]\ar@{-}[uuul]\ar[r]&\bullet \\ v_4&v_5&v_6} \]
Then $\Gamma$ is of elementary type, and we construct it as follows: we start with $\calV_0=\{v_2,v_4,v_6\}$, first we make separately the cones $\nabla(\{v_4\},\varnothing)$ and $\nabla(\{v_6\},\varnothing)$ with tips respectively $v_1$ and $v_3$, then we take the disjoint union of the two cones and of $(\{v_2\},\varnothing)$, finally we make the cone of the disjoint union with tip $v_5$.
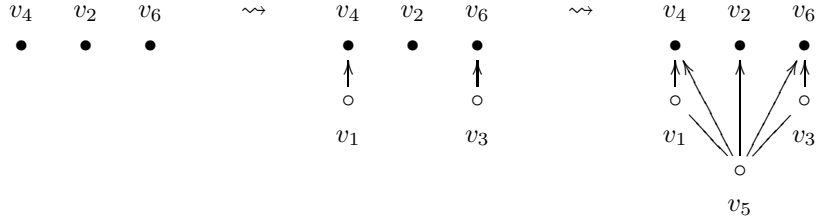
\[
 \xymatrix@R=1.5pt@C=8.5pt{ v_4 & v_2& v_6\\ \bullet&\bullet&\bullet}
 \qquad\ \rightsquigarrow\qquad
 \xymatrix@R=1.5pt@C=8.5pt{ v_4 & v_2& v_6\\ \bullet&\bullet&\bullet \\ \\ \circ\ar[uu]&&\circ\ar[uu] \\ v_1&&v_3}
\qquad\ \rightsquigarrow\qquad
 \xymatrix@R=1.5pt@C=8.5pt{ v_4 & v_2& v_6\\ \bullet&\bullet&\bullet \\ \\
 \circ\ar[uu]&&\circ\ar[uu]  \\ v_1&&v_3\\
 &\circ\ar@{-}[uul]\ar@{-}[uur]\ar[uuuul]\ar[uuuur]\ar[uuuu]&\\ &v_5&}
 \]
\end{exam}

The definition of digraph of elementary type given in the Introduction
is equivalent to Definition~\ref{defin:ET graph}.

\begin{prop}
A digraph $\Gamma=(\calV,\calE)$ satisfies Definition~\ref{defin:ET graph} if, and only if, it is special and for every induced subdigraph $\Gamma'$ of $\Gamma$, either
\begin{itemize}
 \item[(a)] $\Gamma'$ has at least two distinct connected components,
 \item[(b)] or $\Gamma'$ has an ordinary vertex which is adjacent to all other vertices of $\Gamma'$.
\end{itemize}
\end{prop}

\begin{proof}
First, suppose that $\Gamma$ satisfies Definition~\ref{defin:ET graph}.
Since all new vertices added with the procedure described in the definitions are the tips of the cones, and thus ordinary vertices, all special vertices of $\Gamma$ are in the starting set $\calV_0$.
Moreover, these vertices are disjoint, and whenever we add the tip $u$ of a cone, the edge joining $u$ the special vertices in $\calV_0$ are special edges with $u$ as first coordinate.
Hence, every special vertex of $\Gamma$ is a sinkhole, and thus $\Gamma$ is a special digraph.

Now let $\Gamma'=(\calV',\calE')$ be a proper induced subdigraph of $\Gamma$.
We proceed by induction on the number of vertices of $\Gamma$.
Since $\Gamma$ satisfies Definition~\ref{defin:ET graph}, either $\Gamma=\Gamma_1\:\dot\cup\:\Gamma_2$ for two proper subdigraphs $\Gamma_1,\Gamma_2$, or $\Gamma=\nabla(\bar\Gamma)$ for some proper subdigraph $\bar \Gamma$ of $\Gamma$.
Clearly, in both cases the digraphs $\Gamma_1,\Gamma_2$, and $\bar\Gamma$, satisfy Definition~\ref{defin:ET graph}, and they have less vertices than $\Gamma$: therefore, by induction their induced subdigraphs satisfy one of the two conditions~(a)--(b).
\begin{itemize}
 \item If $\Gamma=\Gamma_1\:\dot\cup\:\Gamma_2$, either $\Gamma'$ is the disjoint union of two non-trivial induced subdigraphs respectively of $\Gamma_1$ and $\Gamma_2$ --- and hence $\Gamma'$ satisfies condition~(a) ---, or $\Gamma'$ is an induced subdigraph of $\Gamma_i$ with $i\in\{1,2\}$ --- and hence it satisfies one of the two conditions~(a)--(b) by induction;
 \item if $\Gamma=\nabla(\bar\Gamma)$ with tip $u$, either $u\in\calV'$ --- and hence $\Gamma'$ satisfies condition~(b) with the ordinary vertex $u$ ---, or $\Gamma'$ is an induced subdigraph of $\bar\Gamma$ --- and hence it satisfies one of the two conditions~(a)--(b) by induction.
\end{itemize}

Conversely, suppose that $\Gamma$ is a special digraph, and that every induced subdigraph satisfies one of the two conditions~(a)--(b).
In particular, $\Gamma$ itself satisfies one of the two conditions~(a)--(b).
\begin{itemize}
 \item If it satisfies condition~(a), then $\Gamma=\Gamma_1\:\dot\cup\:\Gamma_2$ for two proper subdigraphs $\Gamma_1,\Gamma_2$;
 \item if it satisfies condition~(b) with an ordinary vertex $u\in\calV$, and if $w\in\calV$ is a special vertex, then $(u,w)\in\calE$ is a special edge, as $w$ is a sinkhole (because $\Gamma$ is a special digraph), and hence $\Gamma=\nabla(\bar\Gamma)$ with tip $u$, where $\bar\Gamma$ is the induced subdigraph of $\Gamma$ whose vertices are $\calV\smallsetminus\{u\}$.
\end{itemize}
In both cases, the digraphs $\Gamma_1,\Gamma_2$, and $\bar\Gamma$, satisfy one of the two conditions~(a)--(b), and we may deconstruct them as done with $\Gamma$.
Altogether, iterating the disassembly of $\Gamma$, we see that $\Gamma$ satisfies Definition~\ref{defin:ET graph}.
\end{proof}

In analogy with special digraphs, one has the following characterization of digraphs of elementary type (cf. \cite[Prop.~2.14]{BQW}).

\begin{prop}\label{prop:ET graphs no subgraphs}
 Let $\Gamma=(\calV,\calE)$ be a special digraph.
 Then $\Gamma$ is of elementary type if, and only if, none of the following occurs as an induced subdigraph:
 \begin{itemize}
  \item[(a)] a graph with geometric realization
  \begin{equation}\label{eq:square line}
 \xymatrix@R=1.5pt{ x&y\\ \circ\ar@{-}[r] & \circ\ar@{-}[ddd] \\ \\ \\
 \circ\ar@{-}[r]\ar@{-}[uuu]  & \circ \\ w&z}
 \qquad\qquad\text{or}\qquad\qquad
 \xymatrix@R=1.5pt{ x&y&z&w \\ \circledast&\circ\ar@{-->}[l]\ar@{-}[r] & \circ\ar@{-->}[r]& \circledast}
 \end{equation}
--- here the two dotted arrows in the right-side diagram mean that $(y,x),(z,w)\in\calE$, and these two edges may be ordinary or special;
\item[(b)] a graph with geometric realization
\begin{equation}\label{eq:Lambdas}
 \xymatrix@R=1.5pt{z & x & y  \\ \circ\ar[r]& \bullet  & \circ\ar[l]}.
\end{equation}
 \end{itemize}
\end{prop}


\section{Oriented pro-$p$ RAAGs}\label{sec:RAAGs}

\subsection{Oriented pro-$p$ RAAGs and special digraphs}
Let $\Gamma=(\calV,\calE)$ be a digraph, and for a $p$-power $q$ --- henceforth we will always tacitly assume that $q=2^f$ with $f\geq2$ if $p=2$ --- let $G$ be the associated oriented pro-$p$ RAAG.
One may define an orientation $\theta_\Gamma\colon G\to1+p\Z_p$ by
\[ \theta_\Gamma(v)=\begin{cases}1+q & \text{if }v\text{ is a sinkhole}, \\
 1 & \text{if }v\text{ is not a sinkhole}.      \end{cases}\]
Hence, $\Img(\theta_\Gamma)\subseteq1+4\Z_2$ if $p=2$.

\begin{rem}\rm
Let $\Gamma=(\calV,\calE)$ be a digraph, and for a $p$-power $q$ let $G$ be the associated oriented pro-$p$ RAAG.
\begin{itemize}
 \item[(a)] There exists an orientation $\theta\colon G\to1+p\Z_p$ such that the natural maps \eqref{eq:def 1cyc} are surjective for every $n\geq1$, if, and only if, $\Gamma$ is a special digraph, and the unique orientation satisfying this property is $\theta_\Gamma$ (cf. \cite[Thm.~4.9]{BQW}).
\item[(b)] If $\Gamma$ is an undigraph, then every associated oriented pro-$p$ RAAG does not depend on the choice of $q$, and it is isomorphic to the pro-$p$ RAAG (i.e., the pro-$p$ completion of the discrete RAAG) associated to $\Gamma$ considered as a simplicial graph.
\end{itemize}
\end{rem}

\begin{lem}\label{lemma:specor RAAG torfree}
 Let $\Gamma=(\calV, \calE)$ be a digraph, and for some $p$-power $q$ let $G$ be the associated oriented pro-$p$ RAAG.
 \begin{itemize}
  \item[(i)] If a vertex $v\in\calV$ is an element of $G$ yielding non-trivial torsion, then $\Gamma$ is not a special digraph.
  \item[(ii)] If $\Gamma$ is a special digraph, $v_1,v_2\in\calV$ are disjoint vertices, and there exists a third vertex $w\in\calV$ which is joint to both $v_1,v_2$, then the subgroup of $G$ generated by $v_1,v_2$ is a 2-generated free pro-$p$ group.
 \end{itemize}

\end{lem}

\begin{proof}
\begin{itemize}
 \item[(i)] Suppose that $\Gamma$ is a special digraph, and let $\Gamma_v=(\{v\},\varnothing)$ be the induced subdigraph of $\Gamma$ whose only vertex is $v$.
The oriented pro-$p$ RAAG associated to $\Gamma_v$ is the free pro-$p$-cyclic group $\langle v\rangle\simeq\Z_p$.
By \cite[Prop.~4.11]{BQW}, the inclusion $\{v\}\hookrightarrow\calV$ induces a monomorphism of pro-$p$ groups $\langle v\rangle\hookrightarrow G$, and therefore $v$ is a torsion-free element of $G$.
\item[(ii)] This statement is \cite[Lemma~6.4]{BQW}.
\end{itemize}
\end{proof}


\subsection{Oriented pro-$p$ RAAGs and torsion}\label{ssec:RAAGs torsion}

The following examples, dealing with oriented pro-$p$ RAAGs associated to digraphs that are not special, will be useful for the proof of Theorem~\ref{thm:main intro}.

\begin{exam}\label{ex:Mennike}\rm
 Let $\Gamma=(\calV,\calE)$ be the non-special digraph with geometric representation
 
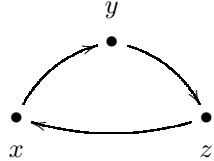
\[  \xymatrix@R=1.5pt{ & y & \\
& \bullet\ar@/^/[dddr] &  \\ \\ \\  \bullet\ar@/^/[uuur] && \bullet\ar@/^/[ll]  \\  x& &z }\]
The oriented pro-$p$ RAAG associated to $\Gamma$ and to a $p$-power $q$ is
\[
 G=\left\langle\:x,y,z\:\mid\:[x,y]=y^q,[y,z]=z^q,[z,x]=x^q,\:\right\rangle,
\]
which is a finite $p$-group, as shown by J.~Mennike (cf. \cite[ Ch.~I, \S~4.4, Ex.~2(e)]{serre:galc}).
In particular, the oriented pro-$p$ RAAG associated to a single vertex (which is isomorphic to $\Z_p$) is not a subgroup of $G$ (cf. \cite[Ex.~4.7]{BQW}).
\end{exam}

\begin{exam}\label{ex:triangletor}\rm
 Let $\Gamma_1=(\calV,\calE_1)$ and $\Gamma_2=(\calV,\calE_2)$ be the non-special digraph with geometric representation respectively
 
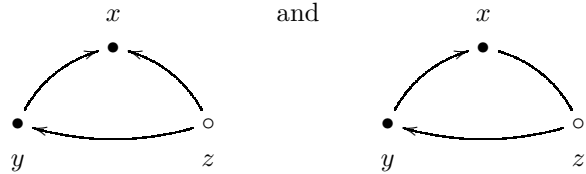
\[  \xymatrix@R=1.5pt{ & x & \\
& \bullet &  \\ \\ \\  \bullet\ar@/^/[uuur] && \circ\ar@/_/[uuul]\ar@/^/[ll]  \\  y& &z }\qquad\text{and}\qquad
\xymatrix@R=1.5pt{ & x & \\
& \bullet &  \\ \\ \\  \bullet\ar@/^/[uuur] && \circ\ar@{-}@/_/[uuul]\ar@/^/[ll]  \\  y& &z }
\]
Given a $p$-power $q$, the oriented pro-$p$ RAAG associated to $q$ and to one of the two graphs is
\[
 G=\left\langle\:x,y,z\:\mid\:[x,y]=y^q,\:[y,z]=z^q,\:[x,z]=z^{\epsilon q}\:\right\rangle,
\]
where $\epsilon=1$ if $G$ is associated to $\Gamma_1$, and $\epsilon=0$ if $G$ is associated to $\Gamma_2$.
Then $[y^q,z]=z^{(1+q)^q-1}$; on the other hand, one computes
\[\begin{split}
\left[[y^q,z]\right]=\left[xyx^{-1}y^{-1},z\right] &=
x\left(y\left(x^{-1}\left(y^{-1}zy\right)x\right)y^{-1}\right)x^{-1}\cdot z^{-1} \\
&= \left(\left(\left(z^{(1+q)^{-1}}\right)^{\epsilon(1+q)^{-1}}\right)^{1+q}\right)^{\epsilon(1+q)}\cdot z^{-1}\\
&=z^1\cdot z^{-1}=1.
\end{split}\]
Since $(1+q)^q-1=(1+q^2+q^3(q-1)/2+\ldots)-1\neq0$, in both cases $z$ yields non-trivial torsion.
In particular, the oriented pro-$p$ RAAG associated to the digraph $\Gamma'=(\{z\},\varnothing)$ is not a subgroup of $G$.
\end{exam}

By Example~\ref{ex:torsion}, the oriented pro-$p$ RAAGs of the above examples are not Frattini-injective nor Frattini-resistant.

\begin{rem}\label{rem:analytic}\rm
One could prove that the oriented pro-$p$ RAAGs in Examples~\ref{ex:Mennike}--\ref{ex:triangletor} are not Frattini-injective using \cite[Thm.~1.2]{ilirslobo:fratres}, as these pro-$p$ groups are {\sl $p$-adic analytic} (cf., e.g., \cite[\S~3]{ilirslobo:fratres}).
\end{rem}

\subsection{Oriented pro-$p$ RAAGs of elementary type}

In analogy with Proposition~\ref{prop:ET graphs no subgraphs}, one may characterize oriented pro-$p$ RAAGs associated to digraphs of elementary type in terms of subgroups (cf. \cite[Prop.~6.3]{BQW}).

\begin{prop}\label{prop:ET RAAGs}
 Let $\Gamma=(\calV,\calE)$ be a special digraph, and for some $p$-power $q$ let $G$ be the associated oriented pro-$p$ RAAG.
 Then $\Gamma$ is of elementary type if, and only if, $G$ has no subgroups isomorphic to the oriented pro-$p$ RAAGs associated to a digraph with geometric realization as in \eqref{eq:Lambdas}, or
  \begin{equation}\label{eq:no ET RAAG}
 \xymatrix@R=1.5pt{ x&y \\ \circ\ar@{-}[r] & \circ\ar@{-}[ddd]  \\ \\ \\
 \circ\ar@{-}[r]\ar@{-}[uuu]  & \circ \\ w&z}
 \qquad\qquad\text{or}\qquad\qquad
 \xymatrix@R=1.5pt{ x&y&z&w \\\circ\ar@{-}[r] & \circ\ar@{-}[r] & \circ\ar@{-}[r]  & \circ}
 \end{equation}
\end{prop}

The main theorem of \cite{BQW} states the following.

\begin{thm}\label{thm:BQW}
 Let $\Gamma=(\calV,\calE)$ be a digraph, and for $q=p^f$ {\rm(}with $f\geq2$ if $p=2${\rm)} let $G$ be the oriented pro-$p$ RAAG associated to $\Gamma$ and to $q$.
 The following are equivalent.
 \begin{itemize}
  \item[(i)] $\Gamma$ is a digraph of elementary type.
  \item[(ii)] Every finitely generated subgroup of $G$ occurs as an oriented pro-$p$ RAAG associated to some digraph.
  \item[(iii)] $G$ may complete into a 1-cyclotomic oriented pro-$p$ group.
  \item[(iv)] $G$ occurs as the maximal pro-$p$ Galois group of a field containing a root of $1$ of order $p$, and $\theta_\Gamma$ coincides with the pro-$p$ cyclotomic character.
 \end{itemize}
\end{thm}

Hence, from Theorem~\ref{thm:BQW} and from Proposition~\ref{prop:1cyc fratres}, one deduces the implication (i)$\Rightarrow$(ii) of Theorem~\ref{thm:main intro}.

\begin{cor}
  Let $\Gamma=(\calV,\calE)$ be a digraph of elementary type, and for $q=p^f$ {\rm(}with $f\geq2$ if $p=2${\rm)} let $G$ be the oriented pro-$p$ RAAG associated to $\Gamma$ and to $q$.
 Then $G$ is Frattini-resistant.
\end{cor}

Also, Corollary~\ref{cor:droms} follows by Theorem~\ref{thm:main intro} and the equivalence (i)$\Leftrightarrow$(ii) of Theorem~\ref{thm:BQW}, while Corollary~\ref{cor:1cyc} follows by Theorem~\ref{thm:main intro} and the equivalence (i)$\Leftrightarrow$(iii)$\Leftrightarrow$(iv) of Theorem~\ref{thm:BQW}.

To prove the implication (ii)$\Rightarrow$(i) of Theorem~\ref{thm:main intro}, we will proceed with a case-by-case analysis, showing --- in the next two sections --- that every induced subdigraph which prevents a digraph from being of elementary type gives rise to pro-$p$ groups that are not Frattini-resistant.



\section{Non-special digraphs and Frattini-resistance}\label{sec:nonspecial}
The goal of this section is to prove the following.

\begin{prop}\label{prop:fratres specially}
Let $\Gamma=(\calV,\calE)$ be a digraph, and for a $p$-power $q$ let $G$ be the
associated oriented pro-$p$ RAAG.
If $G$ is Frattini-resistant, then $\Gamma$ is a special digraph.
\end{prop}

So, let $\Gamma=(\calV,\calE)$ be a digraph.
Recall that the seven possible cases of induced subdigraphs with three vertices that prevent $\Gamma$ from being a special digraph are: the three triangle-graphs yielding torsion
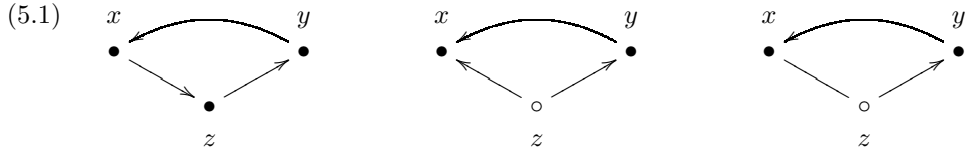
\begin{equation}\label{eq:triangle torsion proof}
  \xymatrix@R=1.5pt{x&&y \\ \bullet\ar[ddr] && \bullet\ar@/_1pc/[ll] \\ \\ &\bullet\ar[ruu]& \\ &z&}
  \qquad\qquad
  \xymatrix@R=1.5pt{x&&y \\ \bullet && \bullet\ar@/_1pc/[ll] \\ \\ &\circ\ar[uur]\ar[luu]& \\ &z&}
  \qquad\qquad
  \xymatrix@R=1.5pt{x&&y \\ \bullet && \bullet\ar@/_1pc/[ll] \\ \\&\circ\ar[ruu]\ar@{-}[uul]&\\ &z&}
\end{equation}
(as shown in Examples~\ref{ex:Mennike}--\ref{ex:triangletor}); the two triangle-graphs
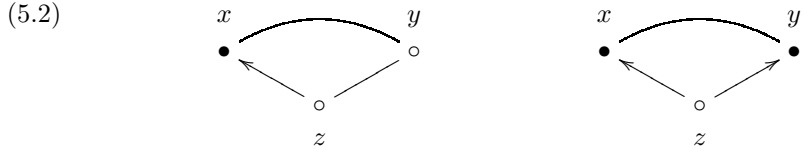
\begin{equation}\label{eq:triangle notorsion proof}
  \xymatrix@R=1.5pt{ x && y \\ \bullet && \circ\ar@{-}@/_1pc/[ll] \\ \\
  &\circ\ar[luu]\ar@{-}[ruu]& \\&z&}
  \qquad\qquad\qquad
  \xymatrix@R=1.5pt{ x && y \\ \bullet && \bullet\ar@{-}@/_1pc/[ll] \\ \\ &\circ\ar[uur]\ar[luu]& \\&z&}
\end{equation}
and the two line-graphs
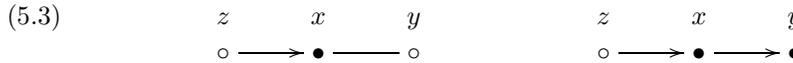
\begin{equation}\label{eq:triangle open proof}
  \xymatrix@R=1.5pt{ z&x&y \\  \circ\ar[r] &\bullet& \circ\ar@{-}[l] }
  \qquad\qquad\qquad
  \xymatrix@R=1.5pt{z&x&y \\  \circ\ar[r] &\bullet\ar[r]& \bullet }
  \end{equation}

  We split the study of non-special digraphs in these three cases.


\subsection{Triangle subdigraphs yielding non-trivial torsion}
 Suppose that $\Gamma$ has an induced subdigraph $\Gamma'=(\calV',\calE')$ as in \eqref{eq:triangle torsion proof}
Let $H$ be the oriented pro-$p$ RAAG associated to $q$ and to $\Gamma'$.
Then by Examples~\ref{ex:Mennike}--\ref{ex:torsion}, $H$ yields non-trivial torsion.
Now consider the homomorphism $H\to G$ induced by the inclusion $\calV'\hookrightarrow \calV$.
Then in all cases $z$ is a non-trivial element of $H$ --- and hence also a non-trivial element of $G$ --- yielding non-trivial torsion.
We have just proved the following.

\begin{prop}\label{prop:torsion RAAGs}
If $\Gamma$ has an induced subdigraph as in \eqref{eq:triangle torsion proof}, then $G$ is not Frattini-resistant.
\end{prop}


\subsection{Non-special triangle digraphs}
Now suppose that $\Gamma$ has an induced subdigraph $\Gamma'=(\calV',\calE')$ as in \eqref{eq:triangle notorsion proof}.
If $G$ has non-trivial torsion, then clearly it is not Frattini-resistant by Example~\ref{ex:torsion}.

Therefore, we may suppose that $G$ is torsion-free.
Then the subgroup of $G$ generated by $x,y$ is a 2-generated free abelian pro-$p$ group --- i.e., it is isomorphic to $\Z_p^2$ ---, and the subgroup generated by $z$ is a 1-generated free abelian pro-$p$ group --- i.e., it is isomorphic to $\Z_p$ --- and it is normalized by $x$ and y.
Altogether, the subgroup of $G$ generated by $\calV'$ is $\langle z\rangle\rtimes\langle x,y\rangle\simeq\Z_p\rtimes\Z_p^2$, and therefore every element $g$ lying in it may be written in a unique way as
\begin{equation}\label{eq:elements triangle}
 g=x^{\lambda_1}y^{\lambda_2}z^{\lambda_3}\quad \text{for some }\lambda_1,\lambda_2,\lambda_3\in\Z_p.\end{equation}
Now set $t=yz$.
Then
\begin{equation}\label{eq:xt triangle}
 [x,t]=xyzy^{-1}x^{-1}z^{-1}= \begin{cases}
z^q & \text{in the first case} \\ z^{q(1+q)} & \text{in the second case}.
 \end{cases}
\end{equation}
Let $H$ be the subgroup of $G$ generated by $x$ and $t$, and let $N$ be the subgroup generated by $t$ and $z^q$.
Then $N$ is isomorphic to $\Z_p^2$ in the first case, and to $\Z_p\rtimes\Z_p$ in the second case, and moreover by \eqref{eq:xt triangle} it is a normal subgroup of $H$.
In particular, one has $H=N\rtimes \langle x\rangle$.
Thus, every element $g$ of $H$ may be written as
\begin{equation}\label{eq:elements triangle2}
 g=x^{\lambda}t^{\mu_1}z^{q\mu_2}=x^\lambda(yz)^{\mu_1}z^{q\mu_2}
 \quad \text{for some }\lambda,\mu_1,\mu_2\in\Z_p.\end{equation}
By \eqref{eq:xt triangle}, one has that $z^{q(1+q)}$, and hence also $z^q$ --- recall that $1+q$ is invertible in $\Z_p$ ---, belongs to $\Phi(H)$.
By \eqref{eq:elements triangle}, $z^{qp^{-1}}$ may not be written as in \eqref{eq:elements triangle2}, and hence it does not belong to $H$.
Therefore, $G$ is not Frattini-resistant.

We have just proved the following.

\begin{prop}\label{prop:nonspecial triangles}
If $\Gamma$ has an induced subdigraph as in \eqref{eq:triangle notorsion proof}, then $G$ is not Frattini-resistant.
\end{prop}

\begin{rem}\rm
The oriented pro-$p$ RAAGs analyzed in this subsection are $p$-adic analytic (cf., e.g., \cite[\S~3]{ilirslobo:fratres} and Remark~\ref{rem:analytic}).
Hence, one may use also \cite[Thm.~1.2]{ilirslobo:fratres}, to show that such an oriented pro-$p$ RAAG is not Frattini-injective, and thus also neither Frattini-resistant.
\end{rem}

\subsection{Line-graphs}
Finally, suppose that $\Gamma$ has an induced subdigraph $\Gamma'=(\calV',\calE')$ as in \eqref{eq:triangle open proof}.
If $G$ has non-trivial torsion, then clearly it is not Frattini-resistant by Example~\ref{ex:torsion}.

Therefore, we may suppose that $G$ is torsion-free.
Let $H$ be the subgroup of $G$ generated by $x,y,z$, let $H_1$, $H_2$, and $V$ be the subgroups of $H$ generated respectively by $x,z$, by $x,y$, and by $y,z$.
 Since we are assuming that $G$ --- and thus also $H$ --- is torsion-free, $H_1$ and $H_2$ are both isomorphic to $\Z_p\rtimes\Z_p$.
 Since
 \[
  yx=\begin{cases} xy &\text{in the 1st case}\\x^{1+q}y &\text{in the 2nd case}\\
     \end{cases}\qquad\text{and}\qquad zx=xz^{(1+q)^{-1}}, \]
every element $g$ of $H$ may be written as $g=x^\lambda u$ for some $u\in V$ and $\lambda\in\Z_p$.

The following lemma is the ``non-special directed analogous'' of Lemma~\ref{lemma:specor RAAG torfree}--(ii), and its proof  follows the strategy of \cite[Lemma~6.4]{BQW}.

\begin{lem}\label{lem:open triangle free subgroup}
 The subgroup $V$ is a 2-generated free pro-$p$ group.
\end{lem}

\begin{proof}
Set $X=\langle x\rangle$.
 Since $G$ is torsion-free, $X\simeq\Z_p$.
 Consider the second-countable pro-$p$ tree $\mathrm{T}=(\calV_{\mathrm{T}},\calE_{\mathrm{T}})$ where
 \[
  \calV_{\mathrm{T}}=\left\{\:gH_1,\:gH_2\:\mid\:g\in H\:\right\},\qquad\calE_{\mathrm{T}}=\left\{\:gX\:\mid\:g\in H\:\right\}.
 \]
Then $H$ acts naturally on $\mathrm{T}$ by $h.(gH_i)=(hg)H_i$, with $i=1,2$, and $h.(gX)=(hg)X$ for every $h\in H$.
The stabilizers in $V$ of an edge $gX$ and of a vertex $gH_i$ are
\[
  \mathrm{Stab}_V(gX)=V\cap {}^gX,\qquad\text{and}\qquad\mathrm{Stab}_V(gH_i)=V\cap{}^gH_i,\qquad i=1,2.
\]
It is straightforward to see that $V\cap {}^gX=\{1\}$, and thus the stabilizer in $V$ of any edge of $\mathrm{T}$ is trivial.
Also, this implies that
$$V\cap{}^gH_1=V\cap\langle{}^gz\rangle\qquad\text{and}\qquad  V\cap{}^gH_2=V\cap\langle{}^gy\rangle,$$
and thus these intersections are either trivial or isomorphic to $\Z_p$.
Since $\mathrm{Stab}_V(gX)=\{1\}$ for any $g\in H$, by \cite[Thm.~5.6]{melnikov} $H$ is isomorphic to the free pro-$p$ product of some stabilizers $\mathrm{Stab}_V(gH_i)$, $i\in\{1,2\}$, and of a free pro-$p$ group.
Hence, altogether $H$ is isomorphic to the free pro-$p$ product of free pro-$p$ groups --- in fact, it cannot be but the free pro-$p$ product of $\mathrm{Stab}_V(H_1)=\langle z\rangle$ and $\mathrm{Stab}_V(H_2)=\langle y\rangle$ ---, and thus it is a free pro-$p$ group.
\end{proof}

Now set $t=zy$.
Then
\begin{equation}\label{eq:xt triangle 2}
 [x,t]= \begin{cases}
[x,z]=z^q & \text{in the first case} \\
xzx^{-1-q}z^{-1}=x^{-q}z^{(1+q)^{1+q}-1} & \text{in the second case}.
 \end{cases}
\end{equation}
Observe that $(1+q)^{1+q}-1$ is an element of $\Z_p$ which is divisible by $q$ but not by $qp$, as
\[\begin{split}
 (1+q)^{1+q} &= 1+\binom{1+q}{1}q+\binom{1+q}{2}q^2+\ldots\\
 &=1+(1+q)q+\frac{q(1+q)}{2}q^2+\ldots
  \end{split}\]
Hence, in the second case $[x,t]=x^{-q}z^{\lambda q}$, where $\lambda$ is an invertible element of $\Z_p$.
Let $K$ be the subgroup of $H$ generated by $x,t$.
By \eqref{eq:xt triangle 2}, $z^q$ belongs to $\Phi(K)$, as $[x,t],x^q\in\Phi(K)$.

We claim that $z^{qp^{-1}}\notin K$.
Let $W$ be the subgroup of $K$ generated by $t,z^q$.
Since $V$ is a free pro-$p$ group and $W\subseteq V$, also $W$ is a free pro-$p$ group, namely, the free pro-$p$ group generated by $t$ and $z^p$.
It is easy to see that every element $g$ of $K$ may be written as
$$g=x^{\lambda}u\qquad\text{for some }u\in W,\:\lambda\in\Z_p,$$ as done for the elements of $H$.
Now, if $z^{qp^{-1}}\in K$, then $z^{qp^{-1}}=x^{\lambda}u$ with $\lambda$ and $u$ as above.
Since $u,z^{qp^{-1}}\in V$, this implies that $x^\lambda\in V$, so that $\lambda=0$, and $z^{qp^{-1}}=u\in W$.
Consequently, $z^q=u^p$, which is impossible as $W$ is the free pro-$p$ group generated by $t$ and $z^p$.

This proves the following.

\begin{prop}\label{prop:nonspecial lambda}
If $\Gamma$ has an induced subdigraph as in \eqref{eq:triangle open proof}, then $G$ is not Frattini-resistant.
\end{prop}

Altogether, Propositions~\ref{prop:torsion RAAGs}--\ref{prop:nonspecial triangles}--\ref{prop:nonspecial lambda} prove Proposition~\ref{prop:fratres specially}.


\section{Digraphs of elementary type and Frattini-resistance}
The goal of this section is to prove the following.

\begin{prop}\label{prop:fratres ET}
Let $\Gamma=(\calV,\calE)$ be a special digraph, and for a $p$-power $q$ let $G$ be the associated oriented pro-$p$ RAAG.
If $G$ is Frattini-resistant, then $\Gamma$ is of elementary type.
\end{prop}

Altogether, Proposition~\ref{prop:fratres ET} and Proposition~\ref{prop:fratres specially} yield implication (ii)$\Rightarrow$(i) of Theorem~\ref{thm:main intro}.

So, let $\Gamma=(\calV,\calE)$ be a special digraph.
Recall that if $\Gamma$ is a special digraph, but not a digraph of elementary type, then by Proposition~\ref{prop:ET graphs no subgraphs} $G$ has a subgroup which is isomorphic to an oriented pro-$p$ RAAG associated to the special digraph \eqref{eq:Lambdas}, or to the oriented pro-$p$ RAAG associated to one of the undigraphs as in \eqref{eq:no ET RAAG}.

As done in \S~\ref{sec:nonspecial}, we proceed analyzing the two cases separately.


\subsection{The special line-digraph}

Suppose that $G$ has a subgroup which is isomorphic to the oriented pro-$p$ RAAG associated to the special digraph \eqref{eq:Lambdas} and to $q$.
Such a pro-$p$ group is
\[
 \left\langle \:x,\:y,\:z\:\mid\:{}^xy=y^{1+q},\:{}^xz=z^{1+q}\:\right\rangle=V\rtimes\langle x\rangle,
\]
where $V=\langle\:y,z\:\rangle$, which is a 2-generated pro-$p$ group by Lemma~\ref{lemma:specor RAAG torfree}--(ii), and it is torsion-free.

If $q=p^f$ with $f\geq2$ let $H$ be the pro-$p$ group generated by $x,y,z$.
Otherwise, if $q=p$, let $H$ be the subgroup of $G$ generated by $y,z$ and by $x^\lambda$ where $\lambda\in\Z_p$ is such that $(1+p)^\lambda=1+p^2$ (such $\lambda$ exists because $1+p$ generates $1+p\Z_p$, and $p$ divides $\lambda$).
We write $\tilde x=x$ and $q'=q$ if $f\geq2$, and $\tilde x=x^\lambda$ and $q'=p^2$ if $q=p$.
Altogether, $H=V\rtimes\langle \tilde x\rangle$, and ${}^xy=y^{1+q'}$, ${}^xz=z^{1+q'}$.

Now set $t=yz^{-1}$ and $v= y^{q'p^{-1}}$, and let $K$ be the subgroup of $H$ generated by $\tilde x,t,v$.
Since
$${}^{\tilde x}t=y^{1+q'}z^{-1-q'}=v^ptz^{-q'},$$
also $z^{q'}\in K$, and moreover
\[
 \left[\tilde x,t\right]=v^ptz^{-q'}t^{-1}=v^p\left[t,z^{-q'}\right]z^{-q'},
\]
so that $z^{q'}$ belongs to $\Phi(K)$, too.

We claim that $z^{q'p^{-1}}\notin K$.
To show this, first we prove the following.

\begin{claim}
The subgroup $W$ of $H$ generated by $v,t,z^{q'}$ is the 3-generated free pro-$p$ group on these three elements.
\end{claim}

\begin{proof}[Proof of the claim]
Since $V$ is a 2-generated free pro-$p$ group, $V$ is the free pro-$p$ group generated by $y$ and $t=yz^{-1}$.

\noindent Clearly, the subgroup $\langle v,t\rangle$ is not pro-$p$-cyclic, as $t$ is not a power of $y$, and vice versa.
 Hence, it is a 2-generated free pro-$p$ group.
 Now suppose that $\langle v,t\rangle=W$, so that $z^{q'}\in \langle v,t\rangle$.
 Then
 \[      z^{q'}=t^{\alpha_1}v^{\beta_1}t^{\alpha_2}v^{\beta_2}\cdots
 =t^{\alpha_1}y^{q'p^{-1}\beta_1}t^{\alpha_2}y^{q'p^{-1}\beta_2}\cdots\]
for some $\alpha_i,\beta_i\in\Z_p$.
On the other hand,
\[
 z^{q'}=(t^{-1}y)^{q'}=\underbrace{t^{-1}yt^{-1}y\cdots t^{-1}y}_{q'\text{ times}},
\]
and thus one has two ways --- which are distinct, as $q'p^{-1}\neq1$ --- to write $z^{q'}$ as an element generated by $y$ and $t$.
But this is a contradiction, as $V$ is the free pro-$p$ group generated by $y$ and $t$.
Thus, $z^{q'}\notin \langle v,t\rangle$, and $W$ is the 3-generated free pro-$p$ group generated by $v,t,z^{q'}$. This proves the claim.
\end{proof}

To conclude, suppose for contradiction that $z^{q'p^{-1}}\in K$.
Since $K=W\rtimes\langle\tilde x\rangle$, one may write $z^{q'p^{-1}}=\tilde x^\lambda w$ for some $\lambda\in\Z_p$ and $w\in W$, and therefore
\[ z^{q'}=\left(z^{q'p^{-1}}\right)^p=\left(\tilde x^\lambda w\right)^p=\tilde x^{p\lambda}w'\qquad\text{for some }w'\in W.\]
Since $z^{q'}\in W$, one deduces that $\lambda=0$, and thus also $z^{q'p^{-1}}\in W$.
But then $z^{q'}$ is a $p$-power of an element of $W$, which contradicts the fact that $W$ is the free pro-$p$ group on $v,t,z^{q'}$.
Hence $z^{q'p^{-1}}\notin K$, and we have just proved the following.

\begin{prop}\label{prop:lambdas}
If $\Gamma$ is not of elementary type and $G$ has a subgroup which is an oriented pro-$p$ RAAG associated to the digraph \eqref{eq:Lambdas}, then $G$ is not Frattini-resistant.
\end{prop}


\subsection{The square-undigraph and line-graphs}

We are left with the two last cases, namely, $G$ contains a subgroup which is the oriented pro-$p$ RAAG associated to one of the two undigraphs \eqref{eq:no ET RAAG}.

Suppose first that $G$ contains a subgroup $H$ which is the oriented pro-$p$ RAAG associated to the square undigraph with vertices $x,y,z,w$ ---
namely,
\[
 \begin{split}
H &=\left\langle\: x,\: y,\:z,\:w\:\mid\: [x,y]=[x,w]=[z,y]=[z,w]=1\:\right\rangle\\
&= \langle\:x,\:z\:\rangle\times\langle\:y,\:w\:\rangle.
 \end{split}\]
Since the two factors are 2-generated free pro-$p$ groups by Lemma~\ref{lemma:specor RAAG torfree}--(ii), none of them is free abelian, and hence $H$ is not Frattini-resistant by Proposition~\ref{prop:directproducts}.

Suppose now that $G$ contains a subgroup which is the oriented pro-$p$ RAAG associated to the line-undigraph with vertices $x,y,z,w$ ---
namely, such a subgroup is
$$\left\langle\: x,y,z,w\:\mid\: [x,y]=[y,z]=[z,w]=1\:\right\rangle$$
--- and let $H$ be the subgroup of $G$ generated by $y,z,{}^xz$ and $t$, where $t=xw$.
Since $tzt^{-1}={}^xz$, $H$ is the HNN-extension of the subgroup $H_1=\langle y,z,{}^xz\rangle$ with $t$, acting as an isomorphism $\langle z\rangle\simeq\langle{}^xz\rangle$ (cf. \cite[Proof of Thm.~3.3]{sz:raags}).
In particular,
\[\begin{split}
    H &= \left\langle\:y,z,{}^zx,t\:\mid\:[y,z]=[y,{}^xz]=1,\:[t,z]={}^xz\cdot z^{-1}\:\right\rangle\\
    &= \left\langle\:y,z,t\:\mid\:[y,z]=[[t,z],y]=1\:\right\rangle.
  \end{split}
\]
Now consider the subgroup $V$ of $H$ (and of $G$) generated by $t,yz,y^p,z^p$.
Then $V$ is the HNN-extension of the subgroup $H_1=\langle y^p,z^p,({}^xz)^p,yz\rangle$ with $t$, acting as an isomorphism $\langle z^p\rangle\simeq\langle({}^xz)^p\rangle$.
One has
\[
  \left[t,z^p\right] = \left({}^xz\right)^p\cdot z^{-p} = \left({}^xz\right)^p\left(y^p\cdot y^{-p}\right)\cdot z^{-p}=
  \left({}^xz\cdot y\right)^p\cdot(y z)^{-p},\]
as $y$ commutes with $z$ and ${}^xz$.
Since $[t,z^p],(yz)^p\in\Phi(V)$, also $({}^xz\cdot y)^p$ lies in $\Phi(V)$.
We claim that ${}^xz\cdot y \notin V$.

Indeed, consider the normal subgroup $N$ of $H$ generated by $y,z$, and set $N_V=N\cap V$: namely, $H/N\simeq V/N_V\simeq\langle t\rangle$.
Then
$$N=\left\langle\:{}^{t^k}y,\:{}^{t^k}z\:\mid\:k\in\Z,\:
\left[{}^{t^k}y,{}^{t^k}z\right]=\left[{}^{t^k}y,{}^{t^{k+1}}z\right]=1\:\right\rangle,$$
and analogously
$$N_V=\left\langle\:{}^{t^k}u\:\mid\:\:k\in\Z,\:
\left[{}^{t^k}u,{}^{t^k}u'\right]=\left[{}^{t^k}y^p,{}^{t^{k+1}}z^p\right]=1,\:u,u'=y^p,z^p,yz\:\right\rangle.$$

The abelianization $N/N'$ is the free abelian pro-$p$ group with basis
\[
 \mathcal{B}_N=\left\{\:{\left[t,_{(k)},y\right]}N',\:{\left[t,_{(k)}z\right]}N'\:\mid\:k\in\Z\ \right\},
\]
where
$$[a,_{(0)}b]=b,\qquad
[a,_{(k)}b]=\underbrace{[a,[a,\ldots[a,[a}_{k\text{ times}},b]]\ldots]]\text{ for }k>0,$$
and $[a,_{(k)}b]=[a^{-1},_{(-k)}b]$ for $k<0$;
and analogously $N_V/N_V'$ is the free abelian pro-$p$ group with basis
\[ \mathcal{B}_{N_V}=\left\{\:{\left[t,_{(k)},u\right]}N_V'\:\mid
 u=y^p,z^p,yz,\:k\in\Z\ \right\}.\]
We underline that $\mathcal{B}_N$ and $\mathcal{B}_{N_V}$ are subsets, of $N/N'$ and $N_V/N_V'$ respectively, converging to 1 --- i.e., any open normal subgroup $U\subseteq H$ contains all but a finite number of elements of $\mathcal{B}_N$ and $\mathcal{B}_{N_V}$ ---, as such a subgroup $U$ contains all commutators of order bigger or equal to $n$, for some $n$ (because $H/U$ is a finite $p$-group, and hence it is nilpotent).

Within the free abelian pro-$p$ group $N/N'$ one has
\begin{equation}\label{eq:yzx mod N}
 {{}^xzy}N'=yN'\cdot zN'\cdot {[t,z]}N',
\end{equation}
and therefore, \eqref{eq:yzx mod N} is the only way to express the coset of ${}^xzyN'$ using the basis $\mathcal{B}_N$.
Now suppose that ${}^xz\cdot y\in V$.
Since ${}^xz\cdot y\in N$, one has ${}^xz\cdot y\in N_V$, and therefore one may write the coset of ${}^xz\cdot y$ in $N_V/N_V'$ as
\begin{equation}\label{eq:xzy mod NV}
{}^xzyN_V'= \prod_{k\in\Z}{[t,_{(k)}y^p]}^{\alpha_k}N_V'\cdot
 {[t,_{(k)}z^p]}^{\beta_k}N_V'\cdot
 {[t,_{(k)}yz]}^{\gamma_k}N_V'
\end{equation}
for some $\alpha_k,\beta_k,\gamma_k\in\Z_p$.
Since $N'\supseteq N_V'$, and since
$$[t,_{(k)}yz]\equiv[t,_{(k)}y][t,_{(k)}z]\qquad\text{and}\qquad[t,_{(k)}u^p]\equiv[t,_{(k)}u]^p$$
modulo $N'$ for every $k\in\Z$ and $u\in N$, from \eqref{eq:xzy mod NV} one obtains the following equality in the free abelian pro-$p$ group $N/N'$:
\[
 {}^xzyN'= y^{p\alpha_0+\gamma_0}N'\cdot z^{p\beta_0+\gamma_0}N'\cdot
 \prod_{k\neq0} {[t,_{(k)}y]}^{p\alpha_k+\gamma_k}N'\cdot
 {[t,_{(k)}z]}^{p\beta_k+\gamma_k}N'.\]
By \eqref{eq:yzx mod N}, one should have
\[
 p\alpha_1+\gamma_1=0\qquad\text{and}\qquad
 p\beta_1+\gamma_1=1,
\]
but then $p(\beta_1-\alpha_1)=1$, a contradiction, as $p\nmid 1$.
Thus, ${}^xz\cdot y\notin V$. This completes the proof of the following.

\begin{prop}\label{prop:fratres square line}
 If $\Gamma$ is not of elementary type and $G$ has a subgroup which is an oriented pro-$p$ RAAG associated to one of the two undigraphs \eqref{eq:no ET RAAG}, then $G$ is not Frattini-resistant.
\end{prop}

Altogether, Propositions~\ref{prop:lambdas}--\ref{prop:fratres square line} give Proposition~\ref{prop:fratres ET}.


\section{Massey products and Frattini-resistant pro-$p$ groups}


\subsection{Triple Massey products and upper-unitriangular representations}\label{ssec:massey unip}
For $n\geq2$ let $\dbU_n$ denote the group of $n\times n$ upper unitriangular matrices with entries in $\Z/p$, namely,
\[
\dbU_n=\left\{\:\left(\begin{array}{ccccc} 1&a_{1,2}&a_{1,3}&\cdots & \\
&1&a_{2,3}& \\ &&\ddots &\ddots&\vdots \\ &&&1&a_{n-1,n} \\ &&&&1 \end{array}\right)\:\mid\:a_{i,j}\in\Z/p\:\right\}\subseteq\mathrm{GL}_n(\Z/p).\]
The center $\Zen(\dbU_n)$ of $\dbU_n$ consists of those matrices whose only non-0 entry --- besides the main diagonal --- is in the top-right corner, namely,
\[
\Zen(\dbU_n)=\left\{\:\left(\begin{array}{ccccc} 1&0&\cdots &0&b \\
&1&0&&0 \\ &&\ddots &\ddots&\vdots \\ &&&1&0 \\ &&&&1 \end{array}\right)\:\mid\:b\in\Z/p\:\right\}.\]
The group $\dbU_n$ is a finite $p$-group, and thus it is also a pro-$p$ group.

Let $G$ be a pro-$p$ group, and consider $\Z/p$ as a trivial $G$-module, as done in \S~\ref{ssec:1cyc}.
Recall that $\rmH^1(G,\Z/p)$ is the $\Z/p$-vector space of all homomorphisms of pro-$p$ groups $G\to\Z/p$ (cf. Remark~\ref{rem:H1}).
Let $\rho\colon G\to \dbU_n$ be a homomorphism, and for $i=1,\ldots,n-1$, let $\rho_{i,i+1}$ denote the projection of $\rho$ onto the $(i,i+1)$-entry of $\rho$.
Then $\rho_{i,i+1}\colon G\to\Z/p$ is a homomorphism, and thus it is an element of $\rmH^1(G,\Z/p)$.
Analogously, if $\bar\rho\colon G\to \dbU_n$ is a homomorphism, and for $i=1,\ldots,n-1$, $\bar\rho_{i,i+1}$ denotes the projection of $\rho$ onto the $(i,i+1)$-entry of $\rho$, then $\bar\rho_{i,i+1}\colon G\to\Z/p$ is a homomorphism, and thus an element of $\rmH^1(G,\Z/p)$.

Given a sequence $\alpha_1,\ldots,\alpha_n$ of length $n$ of (non-necessarily distinct) elements of $\rmH^1(G,\Z/p)$, the subset of $\rmH^2(G,\Z/p)$ which is the value of the Massey product associated to the sequence $\alpha_1,\ldots,\alpha_n$ is denoted by $\langle\alpha_1,\ldots,\alpha_n\rangle$.
One has the following ``pro-$p$ version'' of the result of W.~Dwyer in \cite{Dwyer}
(cf., e.g., \cite[Lemma~9.3]{eq:kummer}, see also \cite[\S~8]{ido:Massey}).

\begin{prop}\label{prop:massey}
Let $G$ be a pro-$p$ group, and for $n\geq2$ let $\alpha_1,\ldots,\alpha_n$ a sequence of length $n$ of (non-necessarily distinct) elements of $\rmH^1(G,\Z/p)$.
\begin{itemize}
 \item[(i)] The $n$-fold Massey product $\langle\alpha_1,\ldots,\alpha_n\rangle$ is not empty if there exists a homomorphism
 $$\bar\rho\colon G\longrightarrow\dbU_{n+1}/\Zen(\dbU_{n+1})$$
 satisfying $\bar\rho_{i,i+1}=\alpha_i$ for all $i=1,\ldots,n$.
 \item[(ii)] The $n$-fold Massey product $\langle\alpha_1,\ldots,\alpha_n\rangle$ vanishes --- i.e., contains 0 --- if there exists a homomorphism
 $$\rho\colon G\longrightarrow\dbU_{n+1}$$
 satisfying $\rho_{i,i+1}=\alpha_i$ for all $i=1,\ldots,n$.
\end{itemize}
\end{prop}

Hence, there are no essential $n$-fold Massey products in the $\Z/p$-cohomology of $G$ if for any sequence $\alpha_1,\ldots,\alpha_n$ of length $n$ of (non-necessarily distinct) elements of $\rmH^1(G,\Z/p)$, either there are no homomorphisms $\bar\rho\colon G\to \dbU_{n+1}/\Zen(\dbU_{n+1})$ as in statement~(i) of Proposition~\ref{prop:massey}, or there is a homomorphism $\rho\colon G\to\dbU_{n+1}$ as in statement~(ii) of Proposition~\ref{prop:massey}.

\begin{rem}\label{rem:massey Heisenberg}\rm
The group $\dbU_3$ is isomorphic to the Heisenberg group modulo $p$, namely
\begin{equation}\label{eq:Heisenberg}
 \dbU_3=\langle\:A,B,C\:\mid\:[A,B]=C,\:A^p=B^p=[A,C]=[B,C]=I_3\:\rangle,
\end{equation}
where
\[
 A=\left(\begin{array}{ccc} 1&1&0 \\ &1&0 \\ &&1 \end{array}\right),\qquad
 \qquad B=\left(\begin{array}{ccc} 1&0&0 \\ &1&1 \\ &&1 \end{array}\right),\qquad
 \qquad C=\left(\begin{array}{ccc} 1&0&1 \\ &1&0 \\ &&1 \end{array}\right).
\]
Now let $G$ be a pro-$p$ group.
Given a sequence $\alpha_1,\alpha_2,\alpha_3$ of length 3 of elements of $\rmH^1(G,\Z/p)$, non--necessarily distinct, then there exists a homomorphism $\bar\rho\colon G\to\dbU_4/\Zen(\dbU_4)$ as in statement~(i) of Proposition~\ref{prop:massey} if, and only if, there exist two homomorphisms $\tau,\tau'\colon G\to \dbU_3$ satisfying
$$\tau(x)\equiv A^{\alpha_1(x)}B^{\alpha_2(x)}\mod \langle\: C\:\rangle \qquad\text{and}\qquad
\tau'(x)\equiv A^{\alpha_2(x)}B^{\alpha_3(x)}\mod \langle\: C\:\rangle$$
for all $x\in G$.
\end{rem}


\subsection{Triple Massey products and Frattini-resistance}\label{ssec:massey conj}

Let $\Gamma=(\calV,\calE)$ be a digraph, and for a $p$-power $q$ let $G$ be the associated oriented pro-$p$ RAAG.
By \cite[Cor.~1.2--(ii)]{BQW}, if $\Gamma$ is of elementary type --- and hence, if $G$ is Frattini-resistant ---, then there are no essential $n$-fold Massey products in the $\Z/p$-cohomology of $G$ for every $n>2$.
In fact, by \cite[Thm.~1.1]{cq:orMassey}, there are many more digraphs whose associated oriented pro-$p$ RAAGs yield no essential $n$-fold Massey products, than digraphs yielding Frattini-resistant pro-$p$ RAAGs: for example, for every undigraph --- and thus also the undigraphs \eqref{eq:no ET RAAG}, which are not of elementary type --- the associated oriented pro-$p$ RAAGs yield no essential $n$-fold Massey products for any $n>2$ (see also \cite[Thm.~1.1]{BCQ}).

Therefore, within the family of oriented pro-$p$ RAAGs associated to digraph, Frattini-resistance is far more restrictive than the absence of essential $n$-fold Massey products for every $n>2$.
This suggests to formulate Conjecture~\ref{conj:massey}, which --- by Proposition~\ref{prop:massey} and Remark~\ref{rem:massey Heisenberg} --- may be formulated in group-theoretic terms as follows.

\begin{conj}[Reformulation of Conjecture~1.4]
Let $G$ be a Frattini-resistant pro-$p$ group, and let $\alpha_1,\alpha_2,\alpha_3$ be (non-necessarily distinct) homomorphisms $G\to\Z/p$.
If there exist two homomorphisms of pro-$p$ groups $\tau,\tau':G\to\dbU_3$ satisfying
$$\tau(x)\equiv A^{\alpha_1(x)}B^{\alpha_2(x)}\mod \langle\: C\:\rangle
\qquad\text{and} \qquad
\tau'(x)\equiv A^{\alpha_2(x)}B^{\alpha_3(x)}\mod\langle\: C\:\rangle$$
for all $x\in G$, then there exists a homomorphism $\rho\colon G\to \dbU_4$
satisfying $\rho_{i,i+1}=\alpha_i$ for $i=1,2,3$.
\end{conj}

\begin{rem}\label{rem:MerSca}\rm
By Proposition~\ref{prop:1cyc fratres}, a positive answer to Conjecture~\ref{conj:massey} would provide a new proof of the recent result of A.~Merkurjev and F.~Scavia \cite[Thm.~1.3]{MerSca3}, which states that pro-$p$ groups which may complete into a 1-cyclotomic oriented pro-$p$ group yield no essential triple Massey products in $\Z/p$-cohomology.
\end{rem}

In fact, we suspect that Conjecture~\ref{conj:massey} may be true not only for triple Massey products, but also for $n$-fold Massey products for every $n\geq3$... but maybe such a question is too daring --- and a positive answer would provide a positive solution to the {\sl Massey vanishing conjecture} for maximal pro-$p$ Galois groups formulated by J.~Mina\v c and N.D.~T\^an, see \cite[Conj.~1.1]{MT:kernel} ---, so we do not formulate it as a conjecture on its own, but we just whisper it.

{\small \subsection*{Acknowledgments}
The author wishes to thank Ilir~Snopce and Slobodan~Tanushevski, for several inspiring discussions on Frattini-resistant pro-$p$ groups (and in particular, for their talks on this topic the author attended).
The author acknowledges his membership to the national group GNSAGA (Algebraic Structures and Algebraic Geometry) of the National Institute of Advanced Mathematics -- a.k.a. INdAM -- ``F. Severi''.
}



\begin{bibdiv}
\begin{biblist}


\bib{BLMS}{article}{
   author={Benson, D.},
   author={Lemire, N.},
   author={Mina\v c, J.},
   author={Swallow, J.},
   title={Detecting pro-$p$-groups that are not absolute Galois groups},
   journal={J. Reine Angew. Math.},
   volume={613},
   date={2007},
   pages={175--191},
   issn={0075-4102},
}

\bib{BCQ}{article}{
   author={Blumer, S.},
   author={Cassella, A.},
   author={Quadrelli, C.},
   title={Groups of $p$-absolute Galois type that are not absolute Galois
   groups},
   journal={J. Pure Appl. Algebra},
   volume={227},
   date={2023},
   number={4},
   pages={Paper No. 107262},
}

\bib{BQW}{article}{
   author={Blumer, S.},
   author={Quadrelli, C.},
   author={Weigel, Th.S.},
   title={Oriented right-angled Artin pro-$\ell$ groups and maximal pro-$\ell$ Galois groups},
   journal={Int. Math. Res. Not.},
   date={2024},
   volume={2024},
   number={8},
   pages={6790--6819},
}


\bib{CasQua}{article}{
   author={Cassella, A.},
   author={Quadrelli, C.},
   title={Right-angled Artin groups and enhanced Koszul properties},
   journal={J. Group Theory},
   volume={24},
   date={2021},
   number={1},
   pages={17--38},
   issn={1433-5883},
}

\bib{diestel}{book}{
   author={Diestel, R.},
   title={Graph theory},
   series={Graduate Texts in Mathematics},
   volume={173},
   edition={5},
   publisher={Springer, Berlin},
   date={2017},
   pages={xviii+428},
}


\bib{ddsms}{book}{
   author={Dixon, J.D.},
   author={du Sautoy, M.P.F.},
   author={Mann, A.},
   author={Segal, D.},
   title={Analytic pro-$p$ groups},
   series={Cambridge Studies in Advanced Mathematics},
   volume={61},
   edition={2},
   publisher={Cambridge University Press, Cambridge},
   date={1999},
   pages={xviii+368},
   isbn={0-521-65011-9},
}
\bib{droms}{article}{
   author={Droms, C.},
   title={Subgroups of graph groups},
   journal={J. Algebra},
   volume={110},
   date={1987},
   number={2},
   pages={519--522},
}

\bib{Dwyer}{article}{
   author={Dwyer, W.G.},
   title={Homology, Massey products and maps between groups},
   journal={J. Pure Appl. Algebra},
   volume={6},
   date={1975},
   number={2},
   pages={177--190},
   issn={0022-4049},
}



\bib{efrat:small}{article}{
   author={Efrat, I.},
   title={Small maximal pro-$p$ Galois groups},
   journal={Manuscripta Math.},
   volume={95},
   date={1998},
   number={2},
   pages={237--249},
   issn={0025-2611},
}

\bib{ido:Massey}{article}{
   author={Efrat, I.},
   title={The Zassenhaus filtration, Massey products, and representations of
   profinite groups},
   journal={Adv. Math.},
   volume={263},
   date={2014},
   pages={389--411},
}


\bib{EM:Massey}{article}{
   author={Efrat, I.},
   author={Matzri, E.},
   title={Triple Massey products and absolute Galois groups},
   journal={J. Eur. Math. Soc. (JEMS)},
   volume={19},
   date={2017},
   number={12},
   pages={3629--3640},
   issn={1435-9855},
}



\bib{eq:kummer}{article}{
   author={Efrat, I.},
   author={Quadrelli, C.},
   title={The Kummerian property and maximal pro-$p$ Galois groups},
   journal={J. Algebra},
   volume={525},
   date={2019},
   pages={284--310},
   issn={0021-8693},
}

%


\bib{eli:Massey}{unpublished}{
   author={Matzri, E.},
   title={Triple Massey products in Galois cohomology},
   date={2014},
   note={Preprint, available at {\tt arXiv:1411.4146}},
}




\bib{melnikov}{article}{
   author={Mel\cprime nikov, O. V.},
   title={Subgroups and the homology of free products of profinite groups},
   language={Russian},
   journal={Izv. Akad. Nauk SSSR Ser. Mat.},
   volume={53},
   date={1989},
   number={1},
   pages={97--120},
   translation={
      journal={Math. USSR-Izv.},
      volume={34},
      date={1990},
      number={1},
      pages={97--119},
      issn={0025-5726},
   },
}


\bib{MerSca3}{unpublished}{
   author={Merkurjev, A.},
   author={Scavia, F.},
   title={On the Massey Vanishing Conjecture and Formal Hilbert 90},
   date={2023},
   note={Preprint, available at {\tt arXiv:2308.13682}},
}

   \bib{MT:kernel}{article}{
   author={Mina\v{c}, J.},
   author={T\^{a}n, N.D.},
   title={The kernel unipotent conjecture and the vanishing of Massey
   products for odd rigid fields},
   journal={Adv. Math.},
   volume={273},
   date={2015},
   pages={242--270},
}


\bib{MT:massey}{article}{
   author={Mina\v{c}, J.},
   author={T\^{a}n, N.D.},
   title={Triple Massey products and Galois theory},
   journal={J. Eur. Math. Soc. (JEMS)},
   volume={19},
   date={2017},
   number={1},
   pages={255--284},
   issn={1435-9855},
}






\bib{cq:noGal}{article}{
     author={Quadrelli, C.},
   title={Two families of pro-$p$ groups that are not absolute Galois
   groups},
   journal={J. Group Theory},
   volume={25},
   date={2022},
   number={1},
   pages={25--62},
   issn={1433-5883},
}


\bib{cq:chase}{article}{
   author={Quadrelli, C.},
   title={Chasing maximal pro-$p$ Galois groups via 1-cyclotomicity},
   journal={Mediterranean J. Math.},
   volume={21},
   date={2024},
   pages={Paper No. 56},
}


\bib{cq:orMassey}{article}{
   author={Quadrelli, C.},
   title={Digraphs, pro-$p$ groups, and Massey products in Galois cohomology},
   date={2024},
   journal={Rocky Mountain J. Math.},
   date={2024},
   note={To appear, available at {\tt arXiv:2403.01464}},
}

\bib{qsv:quadratic}{article}{
				author={Quadrelli, C.},
				author={Snopce, I.},
				author={Vannacci, M.},
				title={On pro-$p$ groups with quadratic cohomology},
				date={2022},
				journal={J. Algebra},
				volume={612},
				pages={636--690},
			}


\bib{qw:cyc}{article}{
   author={Quadrelli, C.},
   author={Weigel, Th.S.},
   title={Profinite groups with a cyclotomic $p$-orientation},
   date={2020},
   volume={25},
   journal={Doc. Math.},
   pages={1881--1916}
   }


\bib{serre:galc}{book}{
   author={Serre, J-P.},
   title={Galois cohomology},
   series={Springer Monographs in Mathematics},
   edition={Corrected reprint of the 1997 English edition},
   note={Translated from the French by Patrick Ion and revised by the
   author},
   publisher={Springer-Verlag, Berlin},
   date={2002},
   pages={x+210},
   isbn={3-540-42192-0},}


\bib{ilirslobo:fratres}{unpublished}{
   author={Snopce, I.},
   author={Tanushevski, S.},
   title={Frattini-injectivity and maximal pro-$p$ Galois groups},
   date={2020},
   note={Preprint, available at {\tt arXiv:2009.09297}},
}

\bib{ilirslobo:products}{article}{
   author={Snopce, I.},
   author={Tanushevski, S.},
   title={Frattini-resistant direct products of pro-$p$ groups},
   journal={J. Algebra},
   volume={622},
   date={2023},
   pages={351--365},
   issn={0021-8693},
}

\bib{sz:raags}{article}{
	author={Snopce, I.},
	author={Zalesski\u{\i}, P.},
	title={Right-angled Artin pro-$p$ groups},
	date={2022},
	journal={Bull. Lond. Math. Soc.},
	volume={54},
	pages={1904-1922},
	number={5},
}



\bib{wurfel}{article}{
   author={W\"{u}rfel, T.},
   title={On a class of pro-$p$ groups occurring in Galois theory},
   journal={J. Pure Appl. Algebra},
   volume={36},
   date={1985},
   number={1},
   pages={95--103},
}



\end{biblist}
\end{bibdiv}
\end{document}